\documentclass[11pt, reqno]{amsart}
 
  \usepackage{color}
 
 \usepackage{epsfig,graphicx,amssymb,amscd,amsmath,amsthm,amstext,amsbsy,upref,amsfonts,enumerate,hyperref}
 
 \usepackage{bbm}
 \usepackage{comment}
 
\usepackage{tikz}
\usetikzlibrary{3d,calc}
\usetikzlibrary{decorations.markings}
\usetikzlibrary{shapes,backgrounds}
\usepackage{pdfpages}

\newcommand{\hbf}{{\bf h}}

\newcommand{\ebf}{{\bf e}}
\newcommand{\fbf}{{\bf f}}
\newcommand{\gbf}{{\bf g}}

\newcommand{\vbf}{{\bf v}}

\def\inpro#1{\left\langle #1 \right\rangle}
\def\Set#1{\left\{\,#1\,\right\}}

\newcommand{\R}{\mathbb R}
\newcommand{\Z}{\mathbb Z}
\newcommand{\N}{\mathbb N}
\newcommand{\C}{\mathbb C}

\newcommand{\cH}{\mathbb{H}}
\newcommand{\cW}{\mathcal{W}}
\newcommand{\cG}{\mathcal{G}}

\newtheorem{definition}{Definition}
\newtheorem{thm}{Theorem}
\newtheorem{proposition}{Proposition}

\newtheorem{corollary}{Corollary}
\newtheorem{ex}{Example}
\newtheorem{com}{Comment}

\begin{document}
 
\title{	Scalability of frames generated by dynamical operators}

\author{Roza Aceska}
\address{Department of Mathematical Sciences, Ball State University,   IN}
\email{raceska@bsu.edu}

\author{Yeon Hyang Kim}
\address{Department of Mathematics, Central Michigan University,   MI}
\email{kim4y@cmich.edu}

\date{2016.}

\begin{abstract}  
Let $A$ be an operator on {a separable   } Hilbert space $\cH$, and let $G \subset \cH$. It is known that - under appropriate conditions on $A$ and $G$ - the set  of iterations  $F_G(A)= \{A^j \gbf \; | \; \gbf \in G,  \; 0 \leq  j \leq  L(\gbf) \} $
is a  frame for $\cH$. We call $F_G(A)$ a dynamical frame for $\cH$, and explore further  its properties; in particular, we show  that the canonical dual frame of   $F_G(A)$ also has an iterative set structure. 

We explore the relations between the operator $A$, the set $G$ and the number of iterations $L$ which ensure that the system $F_G(A)$ is a scalable frame.
We give a general statement on frame scalability, 
and study in detail the   case   when $A$ is a normal operator, utilizing the unitary diagonalization in finite dimensions. In addition, we   answer the question of when $F_G(A)$ is a scalable frame in several special cases involving block-diagonal and companion operators.   
\end{abstract}

 \maketitle

\section{Introduction}  


The problem of generating frames by iterative
actions of operators   \cite{ACMT, ACAMP, AP}   has  emerged within the research related to the dynamical sampling problem \cite{AADP13}-\cite{ACAMP}. 
The 
  conditions   under which a frame  generated  by iterative
  actions of operators
  exists for a finite-dimensional or a separable Hilbert space have been stated in \cite{ACMT} and \cite{ACAMP}.  If we have a frame, then a linear combination of a dual frame with the dynamically sampled coefficients reproduce the original  signal.  
 The natural follow-up questions to ask in this setup are: whether we can obtain a scalable  frame under iterative actions, and if not, whether we can find a dual frame which preserves the dynamical structure.  
 
   Let $A $ be an operator on a separable   Hilbert space $\cH$. We consider   a countable  set of vectors $G $ in $\cH$, and     a function $L : G \rightarrow \Z_+$, where  $\Z_+ = \N \cup \Set{0}$.  Related to the iterated system of vectors   \begin{equation}\label{oursystem}
\{A^j \gbf \; | \; \gbf \in G,  \; 0 \leq  j \leq  L(\gbf) \},\end{equation} 
we answer the following two  questions:

\begin{itemize}
\item[(Q1)]
  What  conditions on   $A$,   $ G$ and   $L$ ensure that   \eqref{oursystem} 
is a scalable  frame for   $\cH$?

\item[(Q2)] Assuming  the system \eqref{oursystem} 
is a  frame for $\cH$,  can we obtain a dual frame  for \eqref{oursystem}, perhaps by iterative actions of some operator? 
\end{itemize}

The   motivation for studying systems of type \eqref{oursystem} comes from  the 
{\it dynamical sampling problem  } 
 (DSP):  Find  sampling locations 
 that allow the reconstruction of an unknown function $\fbf $
 from the scarce 
  samples of $\fbf$,  
  and its
evolved states $A^n \fbf$.
 In the DSP,   $n$   represents  time, and   
$A^*$ is an evolution operator; for instance,   $A^*$ can represent the heat evolution operator,
$\fbf$    the temperature at time
$n
= 0$,
 and
$(A^*)^n \fbf$ the temperature at time
$n$.     The  DSP for  the heat evolution operator was studied in  \cite{LV09, RCLV11}; generalizations of the DSP and related applications  can be found in \cite{AADP13}-\cite{ACAMP}.  

More precisely, the DSP  is as follows: Let the initial state of a dynamical system be represented   by an unknown  element 
$\fbf \in \cH$. 
Say the initial state $\fbf $ is evolving under the action of an 
 operator $A^*$ to the states
$\fbf_j = A^*\fbf_{j-1}$, where $\fbf_0 = \fbf$ and $j \in \Z_+$.  
Given a  set of vectors $G \subset \cH$,  one can  find conditions on $ A$, $G $ and $ L = L(\gbf)$ which
  allow the recovery of the initial state $\fbf$ from the set of samples
$\{ \langle A^{* j}
\fbf , \gbf  \rangle \; | \; \gbf \in G\}_{  j =0}^{L(\gbf)}$. 
%
In short, the problem of signal recovery via dynamical sampling is solvable if the
 set of vectors $F_A^{L}(G) : =\{   A^{  j}
 \gbf     \; | \;    \gbf \in G\}_{j=0}^{  L(\gbf)}$     is  a frame for $\cH$,  \cite{ACMT}.  In frame theory it is known that every frame has  at least one dual frame; if $F_A^{L}(G)$ is a frame for $\cH$, and its dual frame  elements are $\hbf_{\gbf, j}$, then  all $\fbf \in \cH$ are reconstructed as 
\begin{equation}\label{signalreconstructuonDS}
\fbf = \sum_{\gbf \in G} \sum_{j=0}^{L(\gbf)}\langle \fbf,  A^j \gbf \rangle \hbf_{\gbf, j}.
\end{equation}If the frame  $F_A^{L}(G)$  is {\textit{ scalable}}, then its dual frame elements are  $w^2_{j,\gbf} A^j \gbf$ for some {\it scaling coefficients } $w_{j, \gbf}$, and the reconstruction formula \eqref{signalreconstructuonDS} is
\begin{equation}\label{signalreconDSscalable}
\fbf = \sum_{\gbf \in G} \sum_{j=0}^{L(\gbf)}w_{j,\gbf}^2 \langle \fbf,  A^j \gbf \rangle A^j \gbf.
\end{equation}Notice that the frame coefficients in \eqref{signalreconstructuonDS} are exactly the samples
\begin{equation}
\label{sampleseq}
\langle A^{* j} \fbf , \gbf  \rangle=\langle   \fbf , A^j \gbf  \rangle.
\end{equation}    Thus the set of samples $\{ \langle A^{* j}
\fbf , \gbf  \rangle \; | \; \gbf \in G\}_{  j =0}^{L(\gbf)}$  is sufficient for the recovery of $\fbf$. 
  Since \eqref{signalreconstructuonDS} requires   that the  dual frame of  $F_A^{L}(G)$ is known,   unless the frame is scalable as in \eqref{signalreconDSscalable},   it  is    significant to find  the answers to   questions (Q1) and (Q2).

\subsection{Contribution and organization} In Section \ref{prelim} we recall the notions of frames,  scalable frames and, in particular,  frames of iterative actions of operators, i.e., dynamical frames.
 In  Section \ref{allnewstuffbeyongAA},   we illustrate the dynamical  nature  of the canonical dual frame of \eqref{oursystem} in Theorem \ref{canondualframedysam}, and the fusion frame structure of dynamical frames (Corollary \ref{fusiondyn}). In Section \ref{mainresults}   we give a characterization of scalability in Theorem \ref{multiscalablediagonalgen}, under the assumption that   $A$  is   normal.   Section \ref{blockdiagOpsubsection} contains   several generalized examples of frames and scalable frames in lower dimensions,  and we characterize frame scalability in $\R^2$ and $\R^3$. In addition, we provide examples of operators which are not normal, yet generate scalable frames for $\R^2$ and $\R^3$.  
Motivated by  these results, we 
   study block-diagonal operators, which  combine low-dimensional frames into higher-dimensional frames (Theorem \ref{blockresultbig}).  
   %
In Section \ref{compansection}, we also provide examples of  dynamical scalable frames, generated using   companion operators \cite{HJ85} and generalized companion operators. In section \ref{conclusion}   we give initial answers to   question (Q3), addressing frame scalability when multiple operators are involved. 
 
 \section{Preliminaries}\label{prelim}
 
Frames are a generalization of  orthonormal bases. 
For an orthonormal basis $ \{\fbf_i\}_{i \in I}$ of  \( \cH \), it holds 
\begin{equation}\label{onmbrepr}  \fbf = \sum_{i \in I} \inpro{\fbf, \fbf_i} \fbf_i  \;\; \text{ for all }  \fbf \in \cH. \end{equation}
The uniqueness of   representation \eqref{onmbrepr} is not always an advantage. 
In applications such as image and signal processing, the  loss of a single coefficient during  data transmission will prevent the recovery of the original signal, 
unless we ensure redundancy   via   frame spanning.

Since finding a dual frame can be computationally challenging, one significant direction of current research has been on the construction of tight frames in finite dimensions   
\cite{ BM03, STDH07, CMKLT06,  CFHWZ12, HKLW07}.  A tight frame plays the role of its own dual, and provides a reconstruction formula as in  \eqref{onmbrepr} up to a constant.  
 Recently,  the theme of scalable frames has been developed as a method of constructing tight frames from general frames by manipulating the length of frame vectors. 
 Scalable frames  maintain erasure resilience and sparse expansion properties of frames \cite{CC13, CKLMNPS14, KOF13, KOPT13, CKOPW15}. 

   First, let us review relevant  definitions and known results. Throughout this paper $\cH$ denotes a separable Hilbert space. 
 Given an  index set $I$, a sequence $F = \{\fbf_i\}_{i \in I}$ of nonzero elements of   $\cH$ is a \textit{frame} for  $\cH$, if there exist $0<A \leq B < \infty$ such that 
  \begin{equation}
  \label{frameineq}
  A\Vert \fbf \Vert^2 \leq \sum_{i \in I} \vert \langle \fbf , \fbf _i \rangle \vert^2 \leq B\Vert \fbf \Vert^2  \;\; \text{ for all }  \fbf \in \cH.
  \end{equation}
In finite dimensions,
we find  it useful to  express frames as matrices, so we abuse the notation of  $F$ as follows: when $\dim \cH = n$, a frame $F=\{\fbf_i\}_{i \in I}$ for $\cH$ is often represented by a $n \times k$ matrix $F$, whose column vectors are $\fbf_i$, $i = 1, \ldots, k$.
 The frame operator $S = FF^*$ is then   positive, self-adjoint 
 and invertible.
 
 For each frame $F$ there exists at least one \textit{dual} frame $  G= \{\gbf_i\}_{i \in I}$, satisfying
 \begin{equation}
   \label{framerepr}
    \fbf = \sum_{i \in I}  \langle \fbf, \fbf_i \rangle  \gbf_i = \sum_{i \in I}  \langle \fbf, \gbf_i \rangle  \fbf_i \;\; \text{ for all }   \fbf \in \cH.
   \end{equation} 
  The   matrix equation  $ F G^* =     G F  ^* = I$ is an equivalent expression to the  frame representation \eqref{framerepr}. 
The set  $ \{\gbf_i=S^{-1}\fbf_i\}_{i \in I}$ is called the  canonical dual frame.   
    
   Finding a dual frame can be computationally challenging; thus it is of interest to work with tight frames. We say that a frame is $A$-\textit{tight} if    $A=B$ in \eqref{frameineq}. In this case, the function reconstruction is simplified since the frame operator is  the identity operator up to scalar multiplication. 
So,  for an $A$-tight frame, we only need one frame for both analysis and reconstruction, as \eqref{framerepr} becomes
   \begin{equation}
    \fbf =\frac{1}{A} \sum_{i \in I}  \langle \fbf, \fbf_i \rangle  \fbf_i = \frac{1}{A}FF^*\fbf \;\; \text{ for all }   \fbf \in \cH.
   \end{equation} 
 When $A=1$, we call $F$ a Parseval frame.    
 If a frame  $F=\{ \fbf_i\}_{i \in I}$ is not tight, but   we can find scaling coefficients   $w_i \ge 0$, $i \in I$, such that  the scaled frame $F_w=\{w_i\fbf_i\}_{i \in I}$ is   tight, then we call the original frame     $F$ a \textit{scalable} frame. 
 We note that the notion of scalability of a frame is defined for a unit-norm frame in \cite{CKLMNPS14}, but in this manuscript we do not require a scalable  frame to be unit-norm.   
For a scalable frame, the scaled frame representation becomes 
 \begin{equation}\label{defscalablerepr} \fbf =  \sum_{i \in I}  \langle  \fbf,  w_i \fbf_i \rangle w_i \fbf_i =  F_wF_w^*\fbf = F D_{w^2}F ^*\fbf\;\; \text{ for all }   \fbf \in \cH,\end{equation}
 where $D_{w^2}$ denotes a diagonal operator with $w_i ^2$ as diagonal entries. 
  If the scaling coefficients $w_i $ are positive for all $i \in I$, then  we call the original frame     $F$ a \textit{strictly scalable} frame. 

Let $I$ denote a finite or countable index set,   let  $G = \{ \fbf_s \}_{s\in I} \subset \cH$ and let    $A : \cH \to \cH$ be    a bounded  operator. We call the collection 
\begin{equation}\label{orifdynfr}
F_{G}^{\bf L }  (A) = \cup_{s \in I} \{A^j  {\fbf}_{s} \,:\, j=0,1,\ldots,L_s \}
\end{equation} 
 a {\it dynamical system}, where $L_s \geq 0$ ($L_s$ may go to $\infty$) and   ${\bf L}=(L_s)_{s \in I}$ is a   sequence  of iterations.  The operator $A$, involved in generating the set \eqref{orifdynfr}, is sometimes referred to as a {\it dynamical operator}.
   If $A$ is fixed, then we use the notation $F_{G}  ^{\bf L}$, and  if $G =\{  \fbf \}$ and ${\bf L} =\{  L \}$,  then we label \eqref{orifdynfr} by  $F_{\fbf}  ^L$.  
 
 Note that in \cite{ACMT}, $\fbf_{s}$ are chosen to be the standard basis vectors, while in this manuscript, we allow the use of any nonzero vector $\fbf_{s}  \in \cH$.  If  \eqref{orifdynfr} is a frame for $\cH$, then we call $F_{G}^{\bf L }  (A)$ a \textit{dynamical frame}, generated by  operator   $A$, set  $G$ and sequence of iterations ${\bf L}$. 
 \section{New results on dynamical frames}\label{allnewstuffbeyongAA}
   

As we are about to see in Theorem \ref{canondualframedysam}, the canonical dual frame of a dynamical frame preserves the dynamical structure, just like the canonical duals  of wavelet or Gabor frames preserve the corresponding wavelet/Gabor structure \cite{Gro01}.

 \begin{thm}\label{canondualframedysam}
  Let $G = \{ {\fbf}_{s} \}_{s \in I}  \subset\cH$, where $I$ is a countable index set, and assume that  $F_{G}^{\bf L} (A)$    is a frame for $\cH$, with   frame operator $S$.   
 The   canonical dual frame of $F_{G}^{\bf L} (A)$ is the dynamical frame  $F_{ G'} ^{\bf L} (B)$,  generated by   $B=S^{-1}AS$,   $G' =  \{ {\gbf}_s = S^{-1}{\fbf}_{s}  \}_{ s \in I}$, and sequence of iterations $\bf L$.
   That is,   for every $\fbf \in \cH$ the  frame reconstruction formula is 
  \begin{equation}\label{frdynreprAB}
   \fbf = \sum_{s\in I}\sum_{j=0} ^{L_s}\langle  A^{*j}\fbf,    {\fbf}_{s} \rangle  B^j  {\gbf}_{s}.
  \end{equation}
  \end{thm}
  
  \begin{proof}
 The elements of the  canonical dual frame of $F_{G}^{\bf L} (A)$ are computed as  $S^{-1} \left(A^j {\fbf}_{s}\right)$,  $s \in I$, $j=0,1,\ldots,L_s$.  Let ${\gbf}_s = S^{-1}{\fbf}_{s}$,  $s \in I$, then for all $j \geq 0$ we have
 $$B^j {\gbf}_{s} =   (S^{-1}AS)(S^{-1}AS)\ldots(S^{-1}AS){\gbf}_s = S^{-1}A^j \left(S {\gbf}_{s}\right) = S^{-1}\left( A^j    {\fbf}_{s}\right),$$ 
 and \eqref{frdynreprAB} follows by \eqref{framerepr} and \eqref{sampleseq}.
  \end{proof}

 It is a known fact in frame theory  that   an invertible operator preserves the frame inequality. It follows from this that under the action of an invertible operator, the dynamical structure is preserved:     
  \begin{thm}\label{123}
  Let $\cH_1$ and $\cH_2$ be  two separable Hilbert spaces. Let  $G = \{ \fbf_s \}_{s \in I} \subset \cH_1$, where $I$ is a countable index set. Let  ${\bf L}=(L_s)_{s \in I}$,  $L_s \geq 0$. 
  Let $A$ be an operator on $\cH_1$ and let  
    $B: \cH_1 \rightarrow \cH_2$ be an invertible operator. Set  $\gbf_s = B \fbf_s \in \cH_2$, $s \in I$,  and $C = B AB^{-1}$. TFAE:
    \begin{itemize}
     \item[(i)] The set
      $\displaystyle F=\cup_{s\in I}  \{A^j \fbf_{s}\}_{j=0}^{L_s}$ is a  frame for $\cH_1$,
     
     \item[(ii)] The set  $\displaystyle BF= \cup_{s\in I} \{C^j \gbf_{s}\}_{j=0}^{L_s}$ is a  frame for $\cH_2$. 
    \end{itemize}  
    
    

  \end{thm}
  \begin{proof}
 Let $\gbf_s = B\fbf_s \in \cH_2$, $s \in I$,  and set $C = B AB^{-1}$. Note that $C^j = B A^jB^{-1}$, due to $B^{-1}B=I$.  For all $A^j \fbf_s \in F\subset \cH_1$, we have \begin{equation}
  BA^j\fbf_s =BA^j B^{-1}B\fbf_s =BA^j B^{-1} \gbf_s = C^j\gbf_s\in BF\subset \cH_2. \end{equation} The operator $B$ is   invertible, thus $BF$ is a  frame if and only if $F$ is a frame, so (i) and (ii) are equivalent.  
  \end{proof}

  \begin{com}
     If $\cH=\cH_1=\cH_2$, then Theorem \ref{123} is a generalization of the   change of basis result. Notice that  under the action of an invertible operator $B: \cH  \rightarrow \cH $, the elements of a  dynamical frame  $F$ for $\cH$    preserve  the dynamical structure,  i.e.,  $BF$ is also a dynamical frame  for $\cH$. 
  \end{com}
 
 Fusion frames \cite{CK04} are frames which decompose into   
  a union of frames for subspaces of  a Hilbert space $\cH$. 
   Given a countable   index set $I$, let $\cW: = \{W_i \, | \, i \in I \}$ be a family of closed subspaces in $\cH$. Let the orthogonal projections
  onto $W_i$ be denoted by by $P_i$. Then  $\cW$ is a \textit{fusion frame} for $\cH$, if there exist  $C, D >0$ such that
  \[
  C \Vert \fbf \Vert^2 \leq  \sum_{i \in I}    \Vert P_i(\fbf)  \Vert^2  \leq  D\Vert \fbf \Vert^2  \;\; \text{ for all }  \fbf \in \cH. \]
 %
 Let    $F_{i}=\{\fbf_{ij} \}_{ j \in J_i }$ be a frame  for   $W_i$, $i \in I$, with frame bounds $A_i$, $B_i$.  If $0 < A = \inf_{i \in I} A_i \leq \sup_{i \in I} B_i = B < \infty$, then 
     \cite{CK04}: 
 \begin{equation}\label{cassazaff} 
\hspace{-2mm}   \cup_{i \in I} F_{i} \; \text{is a frame for $\cH$ if and only if} \; 
    \{ W_i\}_{i \in I} \; \text{ is a fusion  frame for $\cH$.}
   \end{equation} 
    If  $F_i$  denotes  the frame matrix formed by the frame vectors for each  $W_i$,   and $G_i$ contains the  dual frame elements $\{ \gbf_{ij}\}_{j \in J_i}$, then the fusion frame operator $S$   positive and invertible   on $\cH$, and for all $\fbf \in \cH$, we have
  \begin{equation}
  \label{fusionfrmatrix}
   \fbf  =  \sum_{i \in I} F_iG_i ^* \fbf =  \sum_{i \in I}  G_iF_i ^* \fbf.  
  \end{equation}

  By \eqref{cassazaff} and \eqref{fusionfrmatrix},  a dynamical frame induces a fusion frame:
  
   \begin{corollary}\label{fusiondyn}
   Let     
    $F= \cup_{s\in I} \{A^j  {\fbf}_{s} \}_{  j=0}^{   L_s}$ be a frame for $\cH$. We introduce subspaces of $\cH$ by
   \begin{equation}
   W_s = \overline{span \{A^j  {\fbf}_{ s} \,:\, 0 \leq j \leq L_s \}},\;\; \text{ for all }   s \in I.
   \end{equation}
   Then $\{W_s \}_{s \in I}$ is a fusion frame of $\cH$. 
   \end{corollary}

 \section{  Scalable frames generated by dynamical operators}\label{mainresults}
 
Now, we
study the  scalability of frames of type \eqref{oursystem}.  
 {A prior result  on this topic (see Theorem 8 in \cite{AP}) has restrictive requirements, and delivers a tight frame  if   the involved operator $A$ is a contraction, i.e.,  $A^j \fbf \rightarrow 0$ for all elements $\fbf$ in the studied Hilbert space. } 
Our research results   illuminate the fact  that - in finite dimensions - obtaining a tight or a scalable frame is possible in many   cases.



If the operator $B$ occurring  in Theorem \ref{123} is unitary, then the property of scalability is preserved, and we have:

\begin{corollary}\label{scalabilityinparalel}
Let  $G = \{ \fbf_s \}_{s \in I} \subset \cH$ and ${\bf L}=(L_s)_{s\in I} $, $L_s \geq 0$. 
Let $A$ be a bounded   operator on a separable Hilbert space $\cH$. 
If $B$ is a unitary  operator on $\cH$, 
then   
 $\cup_{s \in I} \{A^j \fbf_{s}\}_{j=0}^{L_s}$ is a scalable  frame if and only if  
   $\cup_{s \in I} \{C^j \gbf_{s}\}_{j=0}^{L_s}$ is a  scalable frame, where  $C =BA B^{*}$ and 
  $\gbf_s = B\fbf_s$, $s \in I$. \end{corollary}
 
 \begin{corollary}\label{generalSchurstatement}
 Let $A, R$ be two operators on a separable Hilbert space $\cH$,  and   let $U$ be a unitary operator on $\cH$. Let   $ \fbf_{s} \in \cH$,    and set $\vbf_s = U^* \fbf_{s}$ for all $s \in I$, where $I$ is a countable index set.  
If $A=URU^*$, then  TFAE:
  
  \begin{itemize}
    \item[(i)]  $\displaystyle  \cup_{s \in I}   \{ A^j \fbf_{ s} \}_{j=0}^{L_s}$ is a scalable frame for $\cH$, 
    \item[(ii)] $\displaystyle  \cup_{s \in I}   \{R^j \vbf_s \}_{j=0}^{L_s}$   { is a scalable frame for $\cH$.}  
  \end{itemize}   
     
     \end{corollary}
  Corollary \ref{generalSchurstatement} is relevant to the {\it Schur}  decomposition: recall, any 
   operator  $A$ on a finite-dimensional Hilbert spaces $\cH$ has a non-unique    Schur   decomposition of type  $A=URU^*$, where $U$ is a unitary $n \times n$ matrix, and $R$ is   of Schur form. 
   When $A=A^*$, i.e., $A$ is normal, then  the Schur decomposition becomes unique, and is reduced to the classical unitary diagonalization. In the next subsection, we exploit the simplicity  of the  unitary diagonalization of  normal operators to give more explicit conditions on the normal operator $A$ in order to ensure scalability of a  frame of type $F^{\bf L} _G(A)$.

\subsection{Normal  operators   }  

Let $A$ be a normal  operator on $\cH$. By the spectral theorem, there exists a unitary operator $U$, and a diagonal  operator $D$ such  that 
$A=UDU^*$; in fact, for each $j \in \Z_+$ $A^j = UD^jU^*$. 

Now, let $\cG = \{ \fbf_s \}_{s \in I}$ and set  $\vbf_s=U^*\fbf_s$, $s \in I$. 
Then for each $j \in \Z_+$,
     \begin{equation}\label{connection}
     A^j \fbf_s = UD^j U^* \fbf_s= UD^j \vbf_s = U(D^j \vbf_s ) 
     \;\; \text{ for all } \fbf_s \in \cG. 
     \end{equation}  
 Corollary \ref{generalSchurstatement} for normal operators reads as follows:
   \begin{corollary}\label{connectSymDiagmulti}
Let $A$ be a normal operator on $\cH$, and  let $A=UDU^*$ be its unitary diagonalization. 
 Let  $\{ \fbf_{s}\}_{ s \in I} \subset \cH$, and set  $\vbf_{s} = U^* \fbf_{s}$, $s \in I$. 
 TFAE

\begin{itemize}
\item[(i)] The set
 $\displaystyle \cup_{ s \in I}\{ A^j \fbf_{s}  | \; j=0,1,\ldots, L_s\}$  is a scalable frame for $\cH$.
  
 \item[(ii)] The set
  $ \displaystyle \cup_{ s \in I} \{ D^j \vbf_{s}  | \; j=0,1,\ldots, L_s\}$ is a scalable frame for $\cH$. 
\end{itemize}
   \end{corollary}  
    

We now  restrict our attention to  a finite dimensional  Hilbert space $\cH =\R^n$ or $\C^n$.  Let us first point out  that the frame scalability property is preserved under  simple manipulations:
  \begin{proposition} Let $F = \{\fbf_i\}_{i=1}^k $ be a  scalable frame for $\cH$, $\dim H = n$. 
  Then the following are also scalable frames: 
 \begin{itemize} 
 \item[(i)]  any column or row permutation of $F $
  
 \item[(ii)]  \(\{ U \fbf_i\}_{i=1}^k\) for any unitary matrix $U$
  \end{itemize}
  \end{proposition}
Given a   diagonal operator $D$ in a  Hilbert space $\cH$ with $\dim \cH = n$, we   first focus our attention on solving    the \textit{one-vector problem}: we look for conditions on  $D$,  and an unknown vector $\vbf \in \cH$, which generate a scalable frame  for $\cH$ of type \eqref{oursystem}. 
  
  Let $L\geq 0$, let $D$ denote a diagonal $n\times n$ matrix, with diagonal entries $a_1,\ldots,a_n$, and let $\vbf = (x(1), \ldots, x(n))^T \in \cH$. 
    Let $ w_j \in \R_+$, $0\leq j \leq L$, be scaling coefficients such that  $ F_W= \{ w_j D^j \vbf   \}_{ j=0} ^{ L} $ is a Parseval frame for $\cH$, i.e., 
   \begin{equation}\label{wantedS}
   F_W F_W^*= I.
   \end{equation}
   Note that \eqref{wantedS} is equivalent to the system of equations
   \begin{eqnarray}\label{wanteddiagscalable}
   |x(i)|^2  \sum_{k=0} ^L w_k ^2   | a_i |^{2k} & =& 1, \quad i=1,\ldots, n; \nonumber \\
   \sum_{k=0} ^L w_k ^2 \left( a_i\bar{a_j }  \right)^k& =& 0,  \quad i \neq j.
   \end{eqnarray}
There exist  real solutions of \eqref{wanteddiagscalable}  when $n \leq 2$.  For instance, when  $\cH = \R^2$, the choice of  $\vbf = (0.5, 0.5)^T$ and  $D= diag(1, -1)$ generates the set  $\{\vbf, D\vbf, D^2 \vbf , D^3\vbf\}$, which is  a Parseval frame for $\R^2$ .
However,  when $\cH = \R^3$,  the equation $\sum_{k=0} ^L w_k ^2 \left( a_ia_j  \right)^k  = 0, \; i \neq j$ implies that for the first three $a_i's$, we always have the relation  $a_1a_2$, $ a_1a_3$, and $a_2a_3$ are all negative numbers assuming $w_i \neq 0, \, i =1, 2, 3$, which is not possible. Thus we have:
    
 \begin{thm}
 Let $\vbf \in \R^n$, and $a_1, \dots, a_n \in \R$. 
 If $n \ge 2$, then any normal operator for $\R^n$ can not generate a strictly scalable frame from $\vbf$. 
 \end{thm} 
     
In contrast to the real case, there exists a solution to the one-vector problem in $\C^n$, involving the  $k$-th root of unity:
        \begin{ex}
  Let $\gamma = e^{2\pi i/ k}$, $k \ge n$. 
  Then the following dynamical operator $A$ and the vector $\vbf$
   \begin{equation*}
   A =   \left( \begin{array}{ccc}
   1& 0& 0  \\
   0 &  \ddots & 0 \\
    0 & 0 & \gamma^{n-1}
     \end{array}\right), \quad 
     \vbf =  \frac{1}{\sqrt{k}}  \left( \begin{array}{c}
   1  \\
   \vdots  \\
    1
     \end{array}\right) 
   \end{equation*}
   generate the Harmonic tight frame $F_{\vbf}^{k-1}$. 
  \end{ex}
   Next, we consider  the multi-generator  case:
  By  \eqref{defscalablerepr}, the scaling coefficients $ w_{s,j}$ related to vectors $D^j \vbf_s$,  $0\leq j \leq L(s)$,  where $\vbf_s =(x_s(1), \ldots, x_s(n))^T$, $1\leq s \leq p$,  need to be solutions to the following system of equations:
   \begin{equation}\label{wantedScalDiagmultivrs}
  \left\{ \begin{array}{ll}
    \sum_{s=1} ^p  |{x_{s}(i)}|^2  \left[w_{s,0}^2+ w_{s,1}^2  |{a_i}|^2 +  \ldots + w_{s,L_s}^2 |{a_i}|^{2L_s} \right]   = 1,  \\
      \sum_{s=1} ^p      x_{s}  (i) \bar{x_{s}}  (j)  \left[  w_{s,0}^2 + w_{s,1}^2 a_i  \bar{a_j} +  \ldots + w_{s,L_s}^2(a_i  \bar{a_j})^{L_k} \right] = 0,  \end{array}\right. \\
     \end{equation} 
     for all $i,j=1,\ldots, n$, $i \neq j$. 
     \begin{proposition}\label{multiscalablediagonalgen} 
   Let $D$ be a diagonal $n\times n$  matrix with 
   diagonal entries $a_1,\ldots, a_n \in \C$, and let $\vbf_s = (x_{s}(1), \ldots, x_{s}(n))^T \in \C^n$, $s \in \{1,\cdots, p\}$, $p \geq 1$.
TFAE:

\begin{itemize}

\item[(i)]    The set $\cup_{s=1} ^p \{ D^j \vbf_s \; | \; j =0,1,\ldots, L_s\}$ is a scalable  frame for $\cH$ 

\item[(ii)] There exist    scaling coefficients   $w_{s,0}, w_{s,1},\ldots, w_{s,L_s}$, $1\leq s \leq p$,  which satisfy conditions \eqref{wantedScalDiagmultivrs}. \end{itemize}
       \end{proposition}

 By Corollary \ref{connectSymDiagmulti} and  Proposition \ref{multiscalablediagonalgen},   the following result holds true for a finite dimensional Hilbert space $\cH$:

\begin{thm}\label{symmetriccaseequivalence}
Let $A = UDU^*$ be a normal $n\times n$ matrix,   where  $U$ is   unitary, and $D$ is  diagonal,  with  diagonal entries $a_1, \ldots, a_n \in \C$.
Let $\fbf_s \in \cH$, and set $\vbf_{s} = U^* \fbf_s = (x_s(1), \ldots, x_s(n))^T$,  $1\leq s \leq p$.

   The set $\cup_{s=1}^p \{ A^j \fbf_{s} \; | \; 0\leq j \leq L_s \}$ is a scalable frame of $\cH$ if an only if there there exists a positive solution   $w_{s,0}, w_{s,1},\ldots, w_{s,L_s}$, $1\leq s \leq p$ to the system of equations \eqref{wantedScalDiagmultivrs}, defined with respect to $a_1, \ldots, a_n$ and $x_s(1), \ldots, x_s(n)$,  $1\leq s \leq p$.
   \end{thm} 
 
 \begin{com}
 The problem of finding   specific conditions under which the set in item (ii) in Corollary \ref{generalSchurstatement}     is a scalable frame for $\cH$  is still open for operators which do not possess a unitary diagonalization. 
  For this reason, we further study  several operators with special structures, such as block-diagonal operators (section \ref{blockdiagOpsubsection}) and companion operators (subsection \ref{compansection}). 
 \end{com}   
    
\section{Block-diagonal operators } \label{blockdiagOpsubsection}
In this section, 
we explore the case when the   operator $A$ is of  block-diagonal form.  Block-diagonal operators give  us a chance to offer a partial answer to (Q1) in the case when we don't have a unitary diagonalization. 

 Note that in subsection \ref{blocks} we give examples of 
   operators which generate scalable frames in Hilbert spaces of dimension $2$ and $3$. Since we can treat $\cH$  with $\dim \cH = n$ as a decomposition of several subspaces of dimensions  2 and 3,   the examples in subsection \ref{blocks} provide   infinite examples of block-diagonal operators which generate scalable frames for $\cH$.   
  \begin{thm}\label{stackScale}
 Let $F_s$ be a scalable  frame for $\cH_{s}$, with $\dim \cH_s = n_s$, $s =1, \ldots p$, and let 
\begin{equation}\label{scal333}
 G =  \left( \begin{array}{ccc}
         F_1 & 0 & 0   \\
              0 & \ddots & 0  \\
              0 &   0 &  F_p
         \end{array}\right). \end{equation}
Then  $G$ is a scalable frame  for 
$\cH= \cH_1 \oplus \ldots \oplus \cH_p$. 
 \end{thm}

\begin{definition}\label{wellembededvr}
Let  $A_s : \cH_s \rightarrow \cH_s$ be an operator on $\cH_s$, with $\dim \cH_s =  n_s$, $1\leq s \leq p$. 
Let $A : \cH_s \rightarrow \cH_s $ be a block-diagonal operator on $\displaystyle \cH=     \oplus_{s=1}^p  \cH_s$, constructed as follows: 
\begin{equation}\label{blockdiagdynamoperator}
A = \left( \begin{array}{ccc}
  A_1 &  \ldots & 0\\
\vdots & \vdots  & \vdots \\
 0 &   \ldots  & A_p
  \end{array} \right).
\end{equation} 
Let  $\vbf   \in \cH_s$ for some $1\leq s \leq p$. We  say that    $\vbf$  is   {\it well-embeded } in     $\fbf \in \cH$   with respect to operator  \eqref{blockdiagdynamoperator} if 
\begin{equation}
\begin{cases} \fbf(j) = \vbf(i), &\mbox{if } j = n_1 +\ldots n_s +i\\ 
\fbf(j) = 0, & \mbox{otherwise.}   \end{cases} 
\end{equation}
 \end{definition}
Whenever $\vbf$ is well-embedded in $\fbf$ with respect to \eqref{blockdiagdynamoperator}, we have 
 $$A\fbf=\left( \begin{array}{c}
   0 \\
 A_s \vbf \\
 0 
  \end{array} \right). $$
 
\begin{thm}\label{blockresultbig} Let  $A_s : \cH_s \rightarrow \cH_s$ be an operator on $\cH_s$, with $\dim \cH_s =  n_s$, $1\leq s \leq p$. 
Let $A : \cH_s \rightarrow \cH_s $ be a block-diagonal operator on $\displaystyle \cH=     \oplus_{s=1}^p  \cH_s$, constructed as in \eqref{blockdiagdynamoperator}. 
Let $\fbf_{s,1}, \ldots, \fbf_{s,m_s} \in \cH$, $1\leq s \leq p$  be   well-embedded vectors $\vbf_{s, 1}   \ldots, \vbf_{s,m_s} \in \cH_s$,  $1\leq s\leq p$.
\begin{equation}\label{bigguy}
\text{The set} \;\;\;\; \;  \bigcup_{s=1}^p \{ A^j \fbf_{s, k} \;\; | \;\;  1\leq k \leq m_s \}_{j=0}^{  L_{s,k} }   \;\;\;\; \; \;\;\;\; \; \;\;\;\; \;
\end{equation}
 { is a (scalable) frame of $\cH$}
   if and only if  $\;  \{ A_s^j \vbf_{s,k} \;  | \;   1\leq k\leq m_s   \}_{j=0}^{  L_{s,k} } $ are (scalable) frames of $\ \cH_s $  for all $1\leq s \leq p$.
 \end{thm}
\begin{proof}
We assume that all $m_s =1$, i.e., $\fbf_{s,k} = \fbf_s$,  $\vbf_{s,k} = \vbf_s$,   and $ L_{s,k}  = L_s$,  $1\leq s \leq p$, to simplify the presentation of the proof. 
The matrix representation of $ \cup_{s=1}^p \{ A^j \fbf_s \}_{j=0}^{L_s}$ with scaling coefficients 
$w_{s, j}$,  $0\leq j \leq L_s$ for each  $s=1,\ldots,  p$ is of block-diagonal form: 
 \begin{equation*}
F=\left( \begin{array}{ccccccc}
 w_{1,0} \vbf_{1}  & \ldots & w_{1, L_1}A_1^{L_1} \vbf_{1}&  &&& \\      
  &&&\ddots &&& \\
  &&& & w_{p, 0}\vbf_{p}  & \ldots & w_{p, L_p} A_p^{L_p}\vbf_{p}   
   \end{array}\right)  . \end{equation*}  
If $F$ is a tight frame, then row vectors of $F$ are orthogonal and have the same norm and so does 
$ (w_{s, 0} \vbf_{s}   \ldots  w_{s, L_s} A_s^{L_s} \vbf_{s})$ for each 
$s =1, \ldots, p$. This implies that the system $  \{ A_s^j  \vbf_{s}\}_{j=0}^{L_k}$ is a scalable frame for  $\cH_s$  for all $1\leq s \leq p$.

Now, suppose that  for each $1\leq s \leq p$, the system $  \{ A^j \vbf_{s}\}_{j=0}^{L_s}$ is a scalable frame for $\cH_s$. Then, there exist some scaling coefficients $w_{s,j}$, $1\leq s \leq p$, $0\leq j\leq L_s$, such that   $ \{ w_{s,j}A_s^{j} \vbf_{s}   | 0\leq j \leq L_s\}$  is a Parseval frame for each $s=1, \ldots p$. 
\end{proof}

 \subsection{Scalable dynamical frames for $\R^2$ and $\R^3$}\label{blocks}
  For the classification of a tight frame in this section,  we use the notion of the {\it diagram vector}. 
For any  \(\fbf  \in\mathbb{R}^n\), we define the diagram vector associated with \(\fbf\), denoted \(\tilde{\fbf}\), by
\begin{equation*}
\tilde{\fbf} = 
\frac{1}{\sqrt{n-1}}
  \left( \begin{array}{c}
\fbf(1)^2-\fbf(2)^2\\  \vdots  \\ \fbf(n-1)^2 -\fbf(n)^2 \\  
\sqrt{2n}\fbf(1)\fbf(2) \\ \vdots  \\  \sqrt{2n}\fbf(n-1)\fbf(n)
\end{array} \right)
\in\mathbb{R}^{n(n-1)\times 1},
\end{equation*}
where the difference of squares 
$\fbf(i)^2- \fbf(j)^2$ and the 
 product \(\fbf(i)\fbf(j)\)  occur exactly once for \(i < j, \ i = 1, 2, \cdots, n-1.\) 
 
 Analogously, for any vector \(\fbf\in\mathbb{C}^n\), we define the diagram vector associated with \(\fbf\), denoted \(\tilde{\fbf}\), by
\begin{equation*}
\tilde{\fbf} = 
\frac{1}{\sqrt{n-1}}
  \left( \begin{array}{c}\fbf(1) \overline{\fbf(1)}-\fbf(2)\overline{\fbf(2)} \\  \vdots  \\ \fbf(n-1)\overline{\fbf(n-1)}-\fbf(n)\overline{\fbf(n)} \\  
\sqrt{n}\fbf(1) \overline{\fbf(2)} \\ \sqrt{n} \overline{\fbf(1)} \fbf(2) \\ \vdots  \\  \sqrt{n}\fbf(n-1)\overline{\fbf(n)} 
\\ \sqrt{n} \overline{\fbf(n-1)} \fbf(n)
 \end{array} \right) \in\mathbb{C}^{3n(n-1)/2},
\end{equation*}
where the difference of the form 
$\fbf(i) \overline{\fbf(i)} - \fbf(j) \overline{\fbf(j)}$ occurs exactly once for \(i < j, \ i = 1, 2, \cdots, n-1\)  and the 
 product of the form \(\fbf(i)  \overline{\fbf(j)} \)  occurs exactly once for \(i  \neq j.\)

The diagram vectors give us the following characterizations of tight frames and scalable frames:
\begin{thm}
\label{charTight}\cite{ CKLMNS13, CKLMNPS14}
Let \(\{\fbf_i\}_{i=1}^k\) be a sequence of vectors in \( \cH \), not all of which are zero. Then \(\{\fbf_i\}_{i=1}^k\) is a tight frame if and only if \(\sum_{i=1}^k\tilde{\fbf_i}=0\). 
\end{thm}

\begin{thm}\label{charScale}\cite{CKLMNS13, CKLMNPS14}
Let  \(\{\fbf_i\}_{i=1}^k\) be  a unit-norm frame for $\cH$ and  $c_1, \cdots, c_k$ be nonnegative numbers, which are not all zero. 
Let $\tilde{G}$ be the Gramian associated to the diagram vectors  \(\{ \tilde{\fbf}_i\}_{i=1}^k\) . 
Then  $\{c_i \fbf_i\}_{i=1}^k $ is a tight frame  for $\cH$ if and only if 
$\fbf = \left(   c_1^2 \ldots  c^2_k  \right)^T$
 belongs to the null space of $\tilde{G}$. 
\end{thm}

 Let $\{\ebf_1, \ldots, \ebf_n\}$ be the standard orthonormal basis in $\R^n$ or $\C^n$. 
 
 \begin{proposition}\label{niceR2example}
 Let $A=   \left( \begin{array}{cc}
    a & c \\
   b & d
    \end{array}\right)$ be an operator in $\R^2$, where $a, b, c, d$ are not all zeros. 
  If $a=0$ and $b \neq 0$, then $F_{\ebf_1}^1 $ is a  scalable frame for $\R^2$.  
 \end{proposition}
 \begin{proof}
 If  $a=0$ and $b\neq 0$,  then $F_{\ebf_1}^1 = \{  (1,0)^T, (0,b)^T\}$. Since the two  vectors in  $F_{\ebf_1}^1 $ are orthogonal,  $F_{\ebf_1}^1 $ is a strictly  scalable frame for $\R^2$. 
 \end{proof}
 We highlight that, when $b=d \neq 0$  and $c=-d/4$     in   Proposition \ref{niceR2example}, the matrix $A$ is non-diagonalizable yet generates a scalable frame for $\R^2$. 
  
 \begin{proposition}\label{2tight}
 Let $a, b, c, d $ be real numbers such that 
  $a \neq -d$,  \[b=   \frac{\pm 1}{a+d}\sqrt{\frac{a^2(a+d)^2 + (a+d)^2 +a^2}{1+(a+d)^2}}, \text{ and }\] 
  \[ c = \mp a(ad+a^2+1) \sqrt{\frac{1+(a+d)^2}{(a+d)^2 +a^2(a+d)^2 +a^2}}.\] 
 Then the   operator $A=   \left( \begin{array}{cc}
    a & c \\
   b & d
    \end{array}\right)$ in $\R^2$ generates a tight frame 
    $$F_{\ebf_1}^2  =\left( \begin{array}{ccc}
    1 & a & a^2 + bc \\
   0 & b & ab+bd
    \end{array}\right)  . $$   
 \end{proposition}
  
 \begin{thm}\label{2scale}
 Let $a, b, c, d $ be real numbers such that 
 $a>0$ and $abcd \neq 0$. Then the following two statements are equivalent:
 \begin{enumerate}
 \item $0< -\frac{ac}{bd} <1$.
 \item The   system 
 $$  F=  \left( \begin{array}{ccc}
   1& a & c \\
   0 & b & d
    \end{array}\right)$$ 
  is a strictly scalable frame for $\R^2$. 
  \end{enumerate}
 \end{thm}
 \begin{proof}
 We first note that the condition 
 $0< -\frac{ac}{bd} <1$
 is equivalent to 
 ($a>0, \,   -\frac{b}{c}> \frac{a}{d}  >0$) or ($a>0, \,   -\frac{d}{a}>  \frac{c}{b} >0$). \\
 $(1)\Rightarrow(2)$:  \quad 
 The conditions  $a>0, \,   -\frac{b}{c}> \frac{a}{d}  >0$ imply that 
 $$ d>0, \,  ad-bc>0,  \,  \frac{ac}{bd} > -1$$ 
 and  the conditions $a>0, \,   -\frac{d}{a}>  \frac{c}{b} >0$ 
  imply that  $$ d<0, \,  ad-bc < 0,  \,  \frac{ac}{bd} > -1. $$ 
 Then 
 $$ x=\sqrt{ \frac{ac}{bd}+1}, \,  y=\sqrt{ \frac{c}{-b(ad-bc)}}, \, z=\sqrt{ \frac{a}{d(ad-bc)}} $$
 are positive numbers  and 
 $$  F=  \left( \begin{array}{ccc}
   x & ya & zc \\
   0 & yb & zd
    \end{array}\right)$$ 
  is a Parseval frame for $\R^2$.  \\
 $(1)\Leftarrow(2)$:  \quad  
 It the system $F$ is strictly scalable, then the normalized system 
 $$F'=\left( \begin{array}{ccc}
   1 & \frac{a}{\sqrt{a^2 + b^2}} & \frac{c}{\sqrt{c^2+d^2}} \\
   0 & \frac{b}{\sqrt{a^2 + b^2}} & \frac{d}{\sqrt{c^2+d^2}}
    \end{array}\right)$$
    is a unit-norm scalable frame. By Theorem \ref{charScale}, the Gramian matrix of diagram vectors of $F'$ has positive scalings in its null space:
 \begin{equation}\label{2x3e1}
 \frac{a^2cd-abc^2+abd^2-b^2cd}{ab(c^2+d^2)}>0,
 \end{equation}
 \begin{equation}\label{2x3e2}
 \frac{-cd(a^2+b^2)}{ab(c^2+d^2)}>0.
 \end{equation}
 Inequality (\ref{2x3e2}) implies that $ -\frac{ac}{bd}>0$. 
 Next we show that $-\frac{ac}{bd}<1$. \\
 In case $b>0$, inequality (\ref{2x3e1}) implies that 
 $$ a^2cd+abd^2 > bc ( ac +bd).$$
 If ($c>0$ and  $ac +bd \ge 0$) or ($c<0$ and  $ac +bd \le 0$), then $ a^2cd+abd^2 >0$, which implies $-\frac{ac}{bd}<1$. 
 If $c>0$ and  $ac +bd < 0$, then $ ac < -bd$, which implies $1 < -\frac{bd}{ac}$ since $ac>0$. 
 Similarly, if $c<0$ and  $ac +bd > 0$, then $ ac > -bd$, which implies $1 < -\frac{bd}{ac}$ since $ac<0$. 
 This is equivalent to $-\frac{ac}{bd}<1$.  \\
 In case $b<0$, suppose that  $-\frac{ac}{bd} \ge 1$. Multiply both sides by the positive number $-abd^2$. On one hand we have 
 $ a^2cd \ge -abd^2 $ and on the other hand,  from inequality (\ref{2x3e1}), we have 
 $a^2cd-abc^2 <-abd^2+b^2cd$. Since $ a^2cd \ge -abd^2 $, we have 
 $-abd^2-abc^2 <-abd^2+b^2cd$, which implies $-\frac{ac}{bd} < 1$. This contradicts  our assumption. 
 \end{proof}
 
 This observation provides us the conditions for a dynamical operator $A$ in $\R^2$ to generate a scalable frame  $F_{\ebf_1}^2  $ for $\R^2$.
 \begin{corollary}\label{2x3scale}
 Let $a, b, c, d $ be real numbers such that 
 $a>0$ and $0< -\frac{a(a^2+bc)}{b^2(a+d)}<1$. 
 Then the  operator $A=   \left( \begin{array}{cc}
    a & c \\
   b & d
    \end{array}\right)$ generates a strictly scalable frame 
      $$F_{\ebf_1}^2   =\left( \begin{array}{ccc}
    1 & a & a^2 + bc \\
   0 & b & ab+bd
    \end{array}\right) .$$
 \end{corollary}
 
 If $2 \sin^2(\omega)-1 >0$, then the  operator 
 $$A= \left( \begin{array}{cc}
            \cos(\omega) & -\sin(\omega) \\
            \sin(\omega)  &  \cos(\omega)
          \end{array}\right)$$
 satisfies the condition on Theorem \ref{2scale}. Consequently we have:
 
 \begin{ex}
 Let 
     $$A= \left( \begin{array}{cc}
            \cos(\omega) & -\sin(\omega) \\
            \sin(\omega)  &  \cos(\omega)
          \end{array}\right),$$
 where $2 \sin^2(\omega)-1 >0$. 
 Then the   operator $A$ generates a strictly scalable  frame 
   $$F_{\ebf_1}^2  = \left( \begin{array}{ccc}
           1&  \cos(\omega) & \cos(2\omega) \\
            0& \sin(\omega)  &  \sin(2\omega)
          \end{array}\right). $$
 
 \end{ex}

  \begin{proposition}\label{2x4scale}
 Let $a, b, c, d $ be real numbers such that $abcd<0$. 
 Then the   system 
 $$  F=  \left( \begin{array}{cccc}
   1 & 0& a & c \\
   0 & 1 & b & d
    \end{array}\right)$$ 
 is a strictly scalable frame for $\R^2$. 
 \end{proposition}
 \begin{proof}
 We define 
  $$ p= \sqrt{ \left( \frac{acd}{b} -c^2 \right) s^2 + 1 },  \,q= \sqrt{ \left(\frac{bcd}{a} +d^2\right) s^2 + 1}, \, r =\sqrt{-\frac{cd}{ab}}. $$
 For any $a, b, c, d$ such that $abcd<0$, one can select $s$ such that  $p>0$ and $q>0$. Those choices of $p, q, r, s$ guarantee that 
 the system 
 $$  F=  \left( \begin{array}{cccc}
   p &0 & ra & sc \\
   0 & q &  rb & sd
    \end{array}\right)$$ 
  is a Parseval frame. 
 \end{proof}
 
 \begin{corollary}\label{2x4scale}
 Let $a, b$ be real numbers such that $a+b^2<0$. 
 Then the   operator $A=   \left( \begin{array}{cc}
    0 & a \\
   1 & b
    \end{array}\right)$ generates a strictly scalable frame $F_{\ebf_1}^3  $ for $\R^2$.   
 \end{corollary}

 We next explore when a dynamical operator $A$ generates a scalable frame $F_{\ebf_1}^3 $ in $\R^3$.  
 We first observe  the following  systems in $\R^3$ when $ab \neq 0$
 \begin{equation}\label{twosystemsdynscal}
  F1 =   \left( \begin{array}{ccccc}
  1 & 0 & 0 & x & y \\
   0 &  1 &  0 & a& c \\
   0 & 0 & 1& b & d \\
   \end{array}\right), \quad 
   F2 =   \left( \begin{array}{ccccc}
  1 & 0 &  x & y \\
   0 &  1 &   a& c \\
   0 & 0 &  b & d \\
   \end{array}\right).
  \end{equation}
  If $F$ is a tight frame, by Theorem \ref{charTight}, we have 
 \begin{equation} \label{onlytwo}
 \begin{array} {ccc}
  ax + cy &=& 0\\
  bx+dy &=&0 \\
  ab+cd&=& 0,
  \end{array}
   \end{equation}
 which implies that $x=y=0$.  
 That is, the last two vectors have only two nonzero elements in the same entries. 
 
 We note that if the first column of $A$ is ${\ebf}_1$, then the system $F_{\ebf_1}^3 $ can not be a frame for $\R^3$. 
 Let 
   \begin{equation}\label{genmatrR3}
  A =   \left( \begin{array}{ccc}
   0 & a & x\\
   1 &  b &  y \\
   0 &   c & z
   \end{array}\right). 
     \end{equation} 
 Then the corresponding $F_4$ system has the following entries:  
 $$
 F_{\ebf_1}^3   =   \left( \begin{array}{cccc}
   1 & 0 & a & ab+cx\\
   0 &  1 &  b & b^2+cy + a \\
   0 &   0 & c & bc +cz
   \end{array}\right). 
 $$
 By  (\ref{onlytwo}), for the system $F_{\ebf_1}^3 $ to be a strictly scalable frame, we need to assume 
 $ a=ab+cx=0$ or $b = b^2+cy + a =0$. 
 We first consider the case $ a=ab+cx=0$. 
 \begin{proposition}\label{prop7import}
 Let $a, b, c, d $ be real  
 numbers such that 
 $a>0$ and $0< -\frac{a(a^2+bc)}{b^2(a+d)}<1$. 
 Then the   operator 
   \begin{equation}\label{nonhermandherm}
  A =   \left( \begin{array}{ccc}
   0 & 0 & 0\\
   1 &  a &  c \\
   0 &   b & d
   \end{array}\right)
     \end{equation} 
      generates a strictly 
      scalable frame
   \begin{equation}\label{26}
  F_{\ebf_1}^3  =   \left( \begin{array}{cccc}
   1 & 0& 0 & 0\\
   0 & 1 & a &   a^2 + bc \\
   0 &  0&  b &   ab+bd
   \end{array}\right).
     \end{equation} 
 \end{proposition}
 \begin{proof}
 This follows from Theorem \ref{stackScale} and Theorem \ref{2scale}.
 \end{proof}
 
 When $b = b^2+cy + a =0$,  we have 
 $$  A =   \left( \begin{array}{ccc}
   0 & a & x\\
   1 &  0 &  -a/c \\
   0 &   c & cz
   \end{array}\right).
 $$
 By applying row and column permutations, $F_{\ebf_1}^3 $ can be written in the same form as (\ref{26}). 
 Similarly, the  following operator, with a suitable choice of the second and third column:
   \begin{equation}\label{genmatrR3}
  A =   \left( \begin{array}{ccc}
   0 & a & x\\
   0 &  b &  y \\
   1 &   c & z
   \end{array}\right)
     \end{equation} 
 generates a scalable frame $F_{\ebf_1}^3 $, which also can be written in the same form as (\ref{26}).


 We note that  any tight or scalable frame in $\R^n$ with $n$ frame vectors is an orthogonal basis.  A trivial example of a scalable dynamical  frame is the following: 
   \begin{ex}\label{examplelemma} Let  
    \begin{equation}\label{companion}
    A  =   \left( \begin{array}{cc}
    0 &  1 \\
    I_{n-1}&  0  \\
      \end{array}\right) .   \end{equation} 
   Then the sequence  $F_{\ebf_1} ^ L  $ is a scalable frame of $\R^n$  if and only if $L \geq n$.
  \end{ex}
 For instance, when $n= L=3$, the resulting  frame is  $F_{\ebf_1}^3 =\{  \ebf_1, \ebf_2,\ebf_3, \ebf_1\}$, and  the scaled frame   $ \{ {2}^{-1/2}\ebf_1, \ebf_2,\ebf_3,2^{-1/2}\ebf_1\}$ is a Parseval frame.

 Notice that \eqref{companion} is an example of a  companion \cite{HJ85}  operator. It makes sense to explore the conditions under which a companion operator  generates a scalable frame.  
 
\section{Companion operators and generalizations}\label{compansection}

Let $a_1, \ldots, a_n \in \R$ which are not all zeros, then 
 \begin{equation}\label{companiondef}
 A  =   \left( \begin{array}{c|c}
   0 &   a_1 \\
   \hline
         &  a_2\\        
   I_{n-1}&  \vdots  \\
    & a_n\\
     \end{array}\right)  \end{equation}
is called a companion operator \cite{HJ85}.

 \begin{proposition}
 Let  the dynamical operator $A$ be a companion operator \eqref{companiondef} in $\R^n$, then  we have 
\begin{enumerate}
\item $ F_{\ebf_1}^{n-1} = I. $
\item  for any orthogonal matrix $U$, the    operator $UAU^{-1}$  generates an orthonormal basis $U$.
\end{enumerate}
  \end{proposition}

It is known that the standard orthonormal basis $B$ can not be extended to a scalable frame by adding one vector $\fbf \in \cH \setminus B$,  
\cite{DKN15, KOF13}. Thus we explore when one can generate a dynamical frame by adding two vectors. 
Although a companion operator $A$ does not generate a scalable frame $F_{\ebf_1}^{n}  $, it can generate a scalable frame $F_{\ebf_1}^{n+1}  $ under certain conditions. 
Using the companion operator $A$, we have 
  \begin{equation}\label{thissystemisframe}
 F_{\ebf_1}^{n}  = (\ebf_1 \ldots \ebf_{n} \, \, \fbf ), \quad 
F_{\ebf_1}^{n+1}   = (\ebf_1 \ldots \ebf_{n} \, \, \fbf  \,\,  \gbf), 
  \end{equation}
  where 
 \[  \fbf= \left( \begin{array}{c}
 a_1 \\
 a_2 \\
 a_3\\
 \vdots\\
   a_{n-1} \\
  a_n 
 \end{array}\right) \text{ and } 
  \gbf= \left( \begin{array}{c}
 a_1a_n \\
 a_1+ a_2a_n \\
 a_2+ a_3a_n\\
 \vdots\\
 a_{n-2} + a_{n-1} a_n\\
 a_{n-1} +a_n^2
 \end{array}\right). 
  \]
Similar calculations as in observation (\ref{onlytwo}) produce the following result: 
\begin{proposition}
\label{ext}
\cite{DKN15} Let $\{{\ebf}_1, \ldots {\ebf}_n\}$ be the standard orthonormal basis in $\R^n$ with $n \ge 2$. Let $\fbf$ and $\gbf$ be two unit-norm vectors in $\R^n$.  

If either system $\{{\ebf}_1, \ldots {\ebf}_n, \fbf, \gbf\}$ or $\{{\ebf}_1, \ldots {\ebf}_{n-1}, \fbf, \gbf\}$ is scalable, then 
$\fbf$ and $\gbf$ have only two nonzero elements in the same entries. 

\end{proposition} 

We now assume that  $F_{\ebf_1}^{n+1}  $  is scalable. Then by Proposition \ref{ext}, 
$a_m=0$ implies that $a_{m-1}=0$ for $m\ge 2$. This implies that $a_1=\ldots = a_{n-2}=0$.

\begin{proposition}\label{companionstandardresult}
Let $a$ and $b$ be real numbers such that 
$a>0$ and $0< -\frac{a^2}{a+b^2}<1$. 
Then the companion operator $A$ in $\R^n$, 
  \begin{equation}\label{gennonherm}
 A =   \left( \begin{array}{cccccc}
 0 & 0 & ...&0& 0 & 0 \\
 1 &  0 & ...&0& 0& 0 \\
  &.&.&.&.& \\
  0 & 0 &    ...&1 &  0  & a\\
 0 & 0 &    ...&0&   1 & b
 \end{array}\right)
  \end{equation}
generates a strictly scalable frame $F_{\ebf_1}^{n+1}  $.  
\end{proposition}
\begin{proof}
 We have 
\begin{equation}\label{nicegen}
    F_{\ebf_1}^{n+1}   =\left( \begin{array}{ccccc}
   I_{n-2} &&&\\
   & 1&  0 & a & ab \\
    & 0& 1& b & a+b^2 \\
       \end{array}\right). 
   \end{equation}
   The strict scalability follows from Theorem \ref{2scale} and Theorem \ref{stackScale}. 
  \end{proof}
     We note that the operator $A$ in (\ref{gennonherm}) is not diagonalizable. Next, we   generalize   the structure of $A$ while ensuring that the new matrix generates  scalable frames by iterative actions. 
   \begin{ex}\label{g1}
Let $a$ and $b$ be real numbers such that   $0< -\frac{a(a^2+bc)}{b^2(a+d)}<1$ and 
$a>0$. 
Then the operator  
 \begin{equation}\label{nicegen}
    A =\left( \begin{array}{ccccc}
    0 & 0 & 0 & \ldots & 0\\
    1 &  0 &  0 &\ldots & 0\\
    0 & 1 & 0 & \ldots & 0\\
     \vdots & \vdots & \vdots &\vdots\\
    0 &  \ldots & 1 & a & c \\
     0 &  \ldots & 0 & b  & d
    \end{array}\right)
   \end{equation}
generates a strictly scalable frame $F_{\ebf_1}^{n}  $ for $\R^n$.

 \end{ex}
 
 \begin{proof}
 
 We have 
\begin{equation}\label{nicegen}
   F_{\ebf_1}^{n}   =\left( \begin{array}{cccc}
   I_{n-2} &&&\\
   & 1& a & a^2+bc \\
    & 0& b & ab+bd \\
       \end{array}\right). 
   \end{equation}
   The strict scalability  follows by Proposition \ref{2tight} and Proposition \ref{stackScale}. 
  \end{proof}

  \begin{ex}
Let $2 \sin^2(\omega)-1 >0$.  Then 
 \begin{equation}\label{nicegen}
    A =\left( \begin{array}{ccccc}
    0 & 0 & 0 & \ldots & 0\\
    1 &  0 &  0 &\ldots & 0\\
    0 & 1 & 0 & \ldots & 0\\
     \vdots & \vdots & \vdots &\vdots\\
    0 &  \ldots & 1 & \cos(\omega) & -\sin(\omega) \\
     0 &  \ldots & 0 &  \sin(\omega)  & \cos(\omega) 
    \end{array}\right)
   \end{equation}
generates a strictly scalable frame $F_{\ebf_1}^{n}$.
 \end{ex}
 
    \begin{ex}
  Let $2 \sin^2(\phi)-1 >0$ and let 
    \begin{equation}\label{realschursimple}
  A = \left( \begin{array}{cccccc}  
  \pm 1 & 0 & 0 & 0  &  \ldots & 0 \\
  0 & \pm 1 & 0 &  0  &  \ldots & 0 \\
  \vdots  & \vdots  & \vdots &   \vdots & \vdots & \vdots \\
  0& 0 & \ldots &  \pm 1&  0 & 0 \\
  0& 0 & \ldots & 0&  \cos \phi &  -\sin \phi \\
  0& 0 & \ldots & 0&  \sin \phi & \cos \phi
   \end{array}\right).\end{equation}
    The set 
   \begin{equation}
   \{  \ebf_{n-1}, A\ebf_{n-1}, A^2\ebf_{n-1}\} \cup \bigcup_{l=1}^{n-2} \{ \ebf_l, A \ebf_l, \ldots, A^{L_l} \ebf_l  \}
   \end{equation}
  is a strictly scalable frame of $\R^n$.  \end{ex}

\section{Concluding remarks and generalizations}\label{conclusion}
We have studied the scalability of dynamical frames in a separable  Hilbert space $\cH$. Given an operator $A$ on $\cH$ and a (at most countable)  set    $G \subset \cH$, 
we have explored the relations between $A$, $G$ and the   number of iterations   that make the system \eqref{oursystem} a scalable frame. When $\dim \cH$ is finite, and  $A$ is a normal operator, we have fully answered question (Q1). 

Since we have not achieved a full answer for   operators which are not unitary diagonalizable, we have offered a partial answer by studying block-diagonal operators,   which are not necessarily normal. Note that the block-diagonal matrix $A$ in Theorem \ref{blockresultbig} cannot be normal if one of its blocks is not normal. Also, we have established the canonical dual frame for  frames of type $F_G(A)$; in particular, we showed that   the canonical dual frame   has, as anticipated, an iterative set structure. This result holds true in any separable Hilbert space $\cH$.

We now pose a new question, which is a generalization of (Q1):

\vspace{2.1mm}

(Q3) \textit{   Given   multiple operators $A_s$, $s \in I$ on a separable Hilbert space $\cH$, and one fixed vector $\vbf \in \cH$, when is the system $\cup_{s \in I} \{ A_s ^j \vbf \}_{j=0}^{L_s}$ a   (scalable)  frame for $\cH$?}

\vspace{2mm}

The next example shows how to generate a scalable frame for $\R^3$ using two dynamical operators. 
    \begin{ex}\label{multigenscalable}  {Let } $\alpha =2\pi/3$,  
\[   A_1 =  \left( \begin{array}{ccc}
  \cos{\alpha } & -\sin{\alpha}  &0\\
   \sin{\alpha}  & \cos{\alpha } &0 \\
   0&0&0
      \end{array}\right), \;  \text{and} \; 
    A_2 =  \left( \begin{array}{ccc}
  \cos{\alpha } & 0  &-\sin{\alpha } \\
  0 &0&0 \\
   \sin{\alpha } &0 & \cos{\alpha}      \end{array}\right).      \]
$$ \text{Then } \; \{ \ebf_1, A_1 \ebf_1, A_1^2  \ebf_1, A_2 \ebf_1, A_2^2  \ebf_1\}  \; \text{is a strictly scalable frame for $\R^3$.}$$
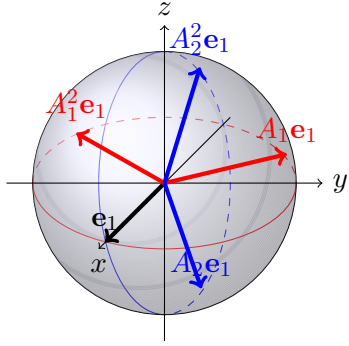
\begin{center}
 \begin{tikzpicture}[scale=1.75]
    \draw [red] (-1,0) arc (180:360:1cm and 0.5cm);
    \draw[red, dashed] (-1,0) arc (180:0:1cm and 0.5cm);
    \draw [blue](0,1) arc (90:270:0.5cm and 1cm);
    \draw[blue, dashed] (0,1) arc (90:-90:0.5cm and 1cm);
    \draw (0,0) circle (1cm);
  \shade[ball color=blue!10!white,opacity=0.20] (0,0) circle (1cm);
    
    \draw[->] (-1.2,0)--(1.2, 0)node[right] {$y$}; 
    \draw[->] (0,-1.2) -- (0,1.2) node[above] {$z$};

     \draw[->] (0.5, 0.5) -- (-0.5, -0.5) node[left, below] {$x$}; 
      \draw[ultra thick,->] (0, 0) -- (-0.45, -0.45) node[above] {${\ebf}_1$}; 
     \draw[red, ultra thick, ->] (0, 0)--(0.93, 0.22 )node[right, above] {$A_1{\ebf}_1$}; 
     \draw[red, ultra thick, ->] (0, 0)--(-0.67, 0.38 )node[right, above] {$A^2_1{\ebf}_1$}; 
     \draw[blue, ultra thick, ->] (0, 0)--(0.28, -0.8 )node[right, above] {$A_2{\ebf}_1$}; 
   \draw[blue, ultra thick, ->] (0, 0)--(0.27, 0.88 )node[right, above] {$A^2_2{\ebf}_1$}; 
    \end{tikzpicture}
  \end{center}  

 \end{ex}

The following proposition is a generalization of the principle introduced in Example \ref{multigenscalable}:
  
 \begin{proposition}
Let $i, j, k,l \in \N$ be  such that $p<k \le n, \, q <l \le n$, and let $N\in \N$. 
For each $m=1, \ldots, N$, we define   $A^{pq}_{kl} (m) =[ a_{ij} (m) ]_{  i, j =1}^n$ as
$$
a_{pq} (m) = a_m,\, 
a_{pl} (m) = b_m,\, 
a_{kq} (m) = c_m,\, 
a_{kl} (m) = d_m.
$$
If for each $m =1, \ldots, N$, $a_m, b_m, c_m$ and $d_m$  satisfy the conditions of Corollary \ref{2scale}, and the system 
\begin{equation} \label{nicesystm}
 \{{\ebf}_1 \} \cup \cup_{m=1}^N \{  A^{pq}_{kl} (m) {\ebf}_1,  (A^{pq}_{kl} (m))^2 {\ebf}_1 \}\end{equation} spans $\R^n$, then 
\eqref{nicesystm} is a strictly scalable frame for $\R^n$. 
\end{proposition}
\begin{proof}
By Corollary \ref{2scale}, the set 
$ \{{\ebf}_1 \} \cup  \{  A^{pq}_{kl} (m) {\ebf}_1,  (A^{pq}_{kl} (m))^2 {\ebf}_1 \} $ is a scalable frame for  a 2-dimensional subspace for each $m =1, \ldots, N$. Thus, there exist  some suitable scaling coefficients  $x(m), y(m), z(m)$, and by Theorem \ref{charTight}, 
$$ \widetilde{x(m){\ebf}_1}   + \widetilde{y(m)A^{pq}_{kl} (m) {\ebf}_1} + \widetilde{z(m) A^{pq}_{kl} (m))^2 {\ebf}_1 }=0 .$$
This implies that the system \eqref{nicesystm}
is a scalable frame for $\R^n$.
\end{proof}
 {For a frame   generated  by iterative actions of  multiple operators, that is, a { \it multi-dynamical} frame, we find that its canonical dual frame is also multi-dynamical:}
 
 \begin{thm}
 Let $A_s$, $s \in I$, be   operators on a separable Hilbert space $\cH$, let $L_s \geq 0$, and  fix a  vector $\vbf \in \cH$. If     $\cup_{s \in I}  \{ A_s ^j \vbf \}_{j=0}^{L_s}$ is a      frame for $\cH$, with frame operator $S$, then its canonical dual frame is 
 \begin{equation}
 \cup_{s \in I}  \{ B_s ^j \fbf \}_{j=0}^{L_s},  
 \end{equation}$\text{where} \; \fbf = S^{-1}\vbf, \; \text{and} \;  B_s= S^{-1}A_s S,  \, s \in I.$
 
 \end{thm}
 \begin{proof}
 If $S$ denotes the frame operator of the frame $\cup_{s}  \{ A_s ^j \vbf \}_{j=0}^{L_s}$   for $\cH$, then its canonical dual frame elements are $S^{-1}A_s ^j \vbf$.  Since  $ B_s^j = S^{-1}A_s^j S$,  we obtain that the dual frame elements are 
  $$S^{-1}A_s ^j \vbf =S^{-1}A_s ^j S S^{-1} \vbf =S^{-1}A_s ^j S \fbf  = B_s^j \fbf.$$
 \end{proof}

 \section*{Acknowledgement}
 We express our gratitude to Professor S. Narayan for many helpful conversations on this work. 
 Kim was supported by the Central Michigan University FRCE Research Type A Grant \#C48143. Aceska was  supported by the BSU Aspire Research Grant ``Frame Theory and Modern Sampling Strategies''.


\begin{thebibliography}{75}
   
\bibitem{AADP13}
R. Aceska, A. Aldroubi, J. Davis, and A. Petrosyan,
Dynamical sampling in shift invariant spaces, 
{\it Contemp. Math. Of the AMS}, 603:139-148, 2013.

\bibitem{AT}
R. Aceska, and S. Tang,
Dynamical Sampling in Hybrid Shift-invariant Spaces,
{\it Contemp. Math. Of the AMS}, 626:   149-166, 2014.

\bibitem{aprstsip} R. Aceska, A.  Petrosyan,   and S. Tang, 
Multidimensional Signal Recovery in Discrete Evolution Systems via Spatiotemporal Trade Off,
{\it  Sampling Theory in Signal and Image Processing},  14(2):  153-169, 2015. 
  
 
\bibitem{apssampta} R. Aceska, A.  Petrosyan,   and S. Tang, 
 Dynamical sampling of two-dimensional temporally-varying signals,
 {\it  Proc. 2015 Int. Conf. Sampling Theory and Applications (SampTA)},  440-443, IEEE, 2015. 
 
  
 
 
 
  
 
 
 

 \bibitem{ADK12}A. Aldroubi, J. Davis and I. Krishtal, { Dynamical Sampling: Time Space Trade-off}, \textit{Appl. Comput. Harmon. Anal.}, {\bf 34}(3), 495--503, 2013. 
 
 \bibitem{ADK13} A. Aldroubi, J. Davis and I. Krishtal, {Exact Reconstruction of Signals in Evolutionary Systems Via Spatiotemporal Trade-off}, \textit{J. Fourier Anal. Appl.}, {\bf 21}(1), 11--31, 2015.
  \bibitem{ACMT}
A. Aldroubi, C. Cabrelli, U. Molter, and Sui Tang,
Dynamical sampling, 
Submitted. Available at http://arxiv.org/abs/1409.8333.
 

 \bibitem{AP}
   A. Aldroubi, and  A. Petrosyan,
 Dynamical sampling and systems from iterative actions of operators,
 Submitted. Available at http://arxiv.org/abs/1606.03136.
 
 \bibitem{ADK12}A. Aldroubi, J. Davis and I. Krishtal, { Dynamical Sampling: Time Space Trade-off}, \textit{Appl. Comput. Harmon. Anal.}, {\bf 34}(3), 495--503, 2013. 
 
 \bibitem{ADK13} A. Aldroubi, J. Davis and I. Krishtal, {Exact Reconstruction of Signals in Evolutionary Systems Via Spatiotemporal Trade-off}, \textit{J. Fourier Anal. Appl.}, {\bf 21}(1), 11--31, 2015. 
 \bibitem{ACAMP}
  A. Aldroubi, C. Cabrelli, A. F. Cakmak, U. Molter,  and A. Petrosyan,
  Iterative actions of normal operators, 
  Submitted. Available at http://arxiv.org/abs/1602.04527.
  
  
 \bibitem {BM03} J. J. Benedetto and M. Fickus, 
 Finite Normalized Tight Frames, 
{\it  Adv. Comput. Math.}, 18:357-385, 2003.
  
   \bibitem{CC13}
J. Cahill and X. Chen, 
A note on scalable frames, 
{\it Proceedings of the 10th International
Conference on Sampling Theory and Applications}, 93-96, 2013.


  
  
   
   
   
  
  \bibitem{CFHWZ12}
 P.G. Casazza,   M.  Fickus, A. Heinecke,  Y. Wang, and  Z. Zhou, 
Spectral tetris fusion frame constructions, 
{\it J. Fourier Anal. Appl.}, 18(4): 828-851, 2012.

\bibitem{CK04} 
  P. G. Casazza and G. Kutyniok,
  Frames of Subspaces, 
  {\it Contemp. Math. Of the AMS}, 345:87-113, 2004.




  
  \bibitem{CKLMNPS14}
 M. Copenhaver, 
Y. Kim, 
C. Logan, 
K. Mayfield, 
S. K. Narayan,  
M. J. Petro, and 
J. Sheperd,  
Diagram vectors and tight frame scaling in finite dimensions, 
{\it Oper. Matrices}, 8(1):78-88, 2014.

 \bibitem{CKLMNS13} 
M. Copenhaver, 
Y. Kim, 
C. Logan, 
K. Mayfield, 
S. K. Narayan,  and 
J. Sheperd,  
Maximum Robustness and  surgery of frames in finite dimensions, 
{\it Linear Algebra Appl.}, 439(5):1330-1339, 2013.

\bibitem{CKOPW15}
X. Chen, G. Kutyniok, K. A. Okoudjou, F. Philipp, and R. Wang, 
 Measures of scalability, 
{\it IEEE Trans. Inf. Theory},  61(8):4410-4423, 2015.


\bibitem{CMKLT06} 
P. Casazza,  M. Fickus, J. Kova\v{c}evi\'{c},  M.T. Leon, and J.C. Tremain, 
A physical interpretation of finite frames, 
{\it Appl. Numer. Harmon. Anal.}, 2-3:51-76, 2006.


\bibitem{DKN15}
R. Domagalski, Y. Kim, and S. K. Narayan,  
On minimal scalings of scalable frames, 
{\it Proceedings of the 11th International
Conference on Sampling Theory and Applications}, 91-95, 2015.

\bibitem{Gro01} 
K. Groechenig,
{\it Foundations of time-frequency analysis}, 
Birkh\"auser Boston, 2001.

\bibitem{HJ85}
R. A. Horn and C. R. Johnson, 
{\it Matrix Analysis}, 
Cambridge University Press, 1985.


\bibitem{HKLW07}
D. Han, K.  Kornelson, D. Larson, and E. Weber, 
{\it Frames for undergraduates}, 
Student Mathematical Library, 40,  American Mathematical Society, Providence, RI, 2007.



\bibitem{KOF13}
 G. Kutyniok, K. Okoudjou, and F. Philipp, 
 Scalable frames and convex geometry, 
{\it Contemp. Math.},  345, 2013.

\bibitem{KOPT13}
G. Kutyniok, K. A. Okoudjou, F.  Philipp, and E. K. Tuley, 
Scalable frames, 
{\it Linear Algebra Appl.}, 438:2225-2238, 2013.

 \bibitem{LV09}
 {\sc Y.~Lu and M.~Vetterli}.   Spatial super-resolution of a diffusion field
   by temporal oversampling in sensor networks, {\it IEEE International Conference on  Acoustics, Speech and Signal
   Processing, 2009. ICASSP 2009.}  
   2249-2252, 2009.
 
 \bibitem{RCLV11}
 {  J.~Ranieri, A.~Chebira, Y.~M. Lu, and M.~Vetterli}.  Sampling and
   reconstructing diffusion fields with localized sources, {\it  IEEE International Conference on  Acoustics, Speech
   and Signal Processing (ICASSP) 2011}, 4016-4019, 2011.
   
 \bibitem{STDH07}
M. A. Sustik,  J. A. Tropp, I. S. Dhillon, R. and W. Heath, 
On the existence of equiangular tight frames, 
{\it Linear Algebra Appl.}, 426:619-635, 2007.
 
 
 
 
   
%
%
%



  
 


 
   \end{thebibliography}
   \end{document}